\documentclass[11pt,a4paper]{elsarticle}

\usepackage[utf8]{inputenc}

\usepackage{amsmath}
\usepackage{amsthm}
\usepackage{enumerate}
\usepackage{mathdots}
\usepackage{amsfonts}
\usepackage{multicol}
\usepackage{amsmath}
\usepackage{amssymb}
\usepackage{fancybox}
\usepackage[usenames,dvipsnames]{color}
\usepackage{graphicx}
\usepackage{tikz}
\usepackage[left=3cm,top=4cm,right=3cm,bottom=3.5cm]{geometry}
\usepackage{arydshln}
\usepackage{ulem}

\newcommand{\efe}{\mathbb{F}}
\newcommand{\FF}{\mathbb{F}}
\newcommand{\F}{\mathbb{F}}

\newcommand{\la}{\lambda}

\def\rank{\mathop{\rm rank}\nolimits}

\newtheorem{theo}{Theorem}[section]

\newtheorem{deff}[theo]{Definition}

\newtheorem{prop}[theo]{Proposition}

\newtheorem{example}[theo]{Example}

\DeclareMathOperator{\diag}{diag}
\DeclareMathOperator{\rev}{rev}

\begin{document}

	\title{Linearizations of matrix polynomials viewed as \\ Rosenbrock's system matrices}

	\author[uc3m]{Froil\'{a}n M. Dopico\fnref{fn1}}
	\ead{dopico@math.uc3m.es}

	\author[upv]{Silvia Marcaida\fnref{fn2}}
	\ead{silvia.marcaida@ehu.eus}
	
	\author[aalto]{Mar\'{i}a C. Quintana\fnref{fn3}}
	\ead{maria.quintanaponce@aalto.fi}
	
	\author[ucl]{Paul~Van~Dooren.\fnref{fn4}}
	\ead{paul.vandooren@uclouvain.be}

	\address[uc3m]{Departamento de Matem\'aticas,
		Universidad Carlos III de Madrid, Avda. Universidad 30, 28911 Legan\'es, Spain.}

	\fntext[fn1]{Supported by grant PID2019-106362GB-I00 funded by MCIN/AEI/ 10.13039/501100011033 and by the Madrid Government (Comunidad de Madrid-Spain) under the Multiannual Agreement with UC3M in the line of Excellence of University Professors (EPUC3M23), and in the context of the V PRICIT (Regional Programme of Research and Technological Innovation).}
	
		\address[upv]{Departamento de Matem\'{a}ticas,
			Universidad del Pa\'{\i}s Vasco UPV/EHU, Apdo. Correos 644, Bilbao 48080, Spain.}
	
		\fntext[fn2]{Supported by grant PID2021-124827NB-I00 funded by MCIN/AEI/ 10.13039/501100011033 and by “ERDF A way of making Europe” by the “European Union”, and by grant GIU21/020 funded by UPV/EHU.}
	
		\address[aalto]{Aalto University, Department of Mathematics and Systems Analysis, P.O. Box 11100, FI-00076, Aalto, Finland.}
	
		\fntext[fn3]{Supported by an Academy of Finland grant (Suomen Akatemian p\"{a}\"{a}t\"{o}s 331230) and by the Agencia Estatal de Investigación of Spain through grant PID2019-106362GB-I00 MCIN/ AEI/10.13039/501100011033/.}
	
		\address[ucl]{Department of Mathematical Engineering, Universit\'{e} catholique de Louvain, Avenue Georges Lema\^itre 4, B-1348 Louvain-la-Neuve, Belgium.}
		
		\fntext[fn4]{Supported by an Aalto Science Institute Visitor Programme}

		\begin{abstract} A well known method to solve the Polynomial Eigenvalue Problem (PEP) is via linearization. That is, transforming the PEP into a generalized linear eigenvalue problem with the same spectral information and solving such linear problem with some of the eigenvalue algorithms available in the literature. Linearizations of matrix polynomials are usually defined using unimodular transformations. In this paper we establish a connection between the standard definition of linearization for matrix polynomials introduced by Gohberg, Lancaster and Rodman and the notion of polynomial system matrix introduced by Rosenbrock. This connection gives new techniques to show that a matrix pencil is a linearization of the corresponding matrix polynomial arising in a PEP.

	\end{abstract}


	\maketitle
	
	\section{Introduction}

	Let $\F$ be an arbitrary field, and let $\overline{\F}$ be the algebraic closure of $\FF$. $\efe[\la]$ denotes the ring of polynomials with coefficients in $\efe,$ and $\efe(\la)$ the field of rational functions over $\F[\la]$. The sets of $p\times m$ matrices with elements in $\efe,$ $\efe[\la]$ and $\efe(\la)$ are denoted by $\efe^{p\times m}$, $\efe[\la]^{p\times m}$ and $\efe(\la)^{p\times m}$, respectively. The elements of $\efe[\la]^{p\times m}$ are called matrix polynomials or polynomial matrices, and the elements of $\efe(\la)^{p\times m}$ are called rational matrices.
	
	A polynomial matrix $P(\lambda)\in\efe[\lambda]^{p\times m}$ can always be written in the form
	\begin{equation}\label{eq:matrix poly}
	P(\lambda) =  P_k \lambda^k +  P_{k-1} \lambda^{k-1} + \cdots + P_1  \lambda + P_0,
	\end{equation}
	where  $P_k,\hdots,P_1,P_0\in\mathbb{F}^{p\times m}$ with $P_k\neq 0$.
	The scalar $k$ is then called the degree of $P(\lambda)$, and it is denoted by $\deg P(\lambda)$. Polynomial matrices of degree $1$ or $0$, i.e., linear polynomial matrices, are called pencils.
	
	The (finite) eigenvalues of a polynomial matrix $P(\la)\in\F[\la]^{p\times m}$ are defined as the scalars $\la_0\in\overline{\F}$ such that $$\rank P(\la_0)< \displaystyle\max_{\mu\in\overline{\F}}\rank P(\mu).$$ The Polynomial Eigenvalue Problem (PEP) consists of finding the eigenvalues of $P(\la).$ If $P(\la)\in\F[\la]^{m\times m}$ is regular (i.e., square with $\det P(\la)\not\equiv 0$), the PEP is equivalent to the problem of finding scalars $\la_0\in\overline{\F}$ such that there exist nonzero constant vectors $x\in\overline{\efe}^{m\times 1}$ and $y\in\overline{\efe}^{m\times 1}$ satisfying $$P(\la_0)x=0 \quad\mbox{and}\quad y^{T}P(\la_0)=0,$$
	respectively.
	The vectors $x$ are called right eigenvectors associated with $\la_0$, and the vectors $y$ are called left eigenvectors associated with $\la_0$.
	
To solve the PEP and to develop a theory of regular polynomial matrices,  Gohberg, Lancaster and Rodman \cite{GohbergLancasterRodman09} introduced in the eighties the notion of \textit{linearization} of a matrix polynomial. Given a matrix polynomial $P(\lambda)$ of degree $k >1$, a linearization of $P(\la)$ is a pencil $L(\lambda):= \la L_1 +  L_0 $ such that there exist unimodular matrices (i.e., square polynomial matrices with nonzero constant determinant) $U_1(\lambda)$ and $V_1(\lambda)$ satisfying
	\begin{equation}\label{eq_linearization}
	U_1(\lambda)L(\lambda)V_1(\lambda) =
	\begin{bmatrix}
	P(\lambda) & 0 \\
	0 & I_s
	\end{bmatrix},
	\end{equation}
	where $I_s$ denotes the identity matrix of size an integer $s\geq 0$. It is known that $L(\la)$ has the same finite eigenvalues with the same partial multiplicities as $P(\la)$. $L(\la)$ is said to be a strong linearization of $P(\la)$ if, in
	addition, there exist unimodular matrices $U_2(\lambda)$ and $V_2(\lambda)$ satisfying
	\begin{equation}\label{eq_strong_lin}
	U_2(\lambda)\rev_1 L(\la)V_2(\lambda) =
	\begin{bmatrix}
	 \rev_{\ell} P(\la)& 0 \\
	0 & I_s
	\end{bmatrix},
	\end{equation}
	where $ \rev_1 L(\la):= \la L_0 + L_1 $ and	$\rev_{\ell} P(\la) :=\la^{\ell} P(1/\la)$ for some $\ell \geq \deg P(\la)$ \cite{GohbergKaashoekRodman88}. In this case, $L(\la)$ has the same finite and infinite eigenvalues with the same partial multiplicities as $P(\la)$.  Here we remark that an equivalent definition of strong linearization is obtained if the matrices $U_2(\lambda)$ and $V_2(\lambda)$ in \eqref{eq_strong_lin} are required to be just rational matrices invertible at $0$ (i.e., $U_2(0)$ and $V_2(0)$  are invertible) instead of unimodular (see \cite[Section 3]{strong} or the proof of Proposition \ref{prop_strong_lin}). In that case we say that $ \rev_1 L(\la)$ and $\diag ( \rev_{\ell} P(\la),I_s)$ are equivalent at $0$.
	
In this paper the terms linearization and strong linearization of a matrix polynomial always refer to the Gohberg, Lancaster and Rodman's definitions in \eqref{eq_linearization} and \eqref{eq_strong_lin}, though other non-equivalent definitions of linearizations are available in the literature \cite{strongminimal,stronglyminimalSIMAX}. In the last two decades, definitions \eqref{eq_linearization} and \eqref{eq_strong_lin} have been very influential in many families of linearizations that have been developed with the goal of solving unstructured and structured PEPs (see, for instance, \cite{aclan-2009,greeks,fiedler_structured,fiedler_structured2,dtdm-fiedler,BKL, htm,lawrence-perez,mmmmVS,mmmmGOOD,noferini-perez,greeks2} among many other references on this topic). In fact, in most of these references, the proofs that certain families of pencils are linearizations and strong linearizations of a polynomial matrix $P(\la)$ are performed by finding explicitly the unimodular matrices $U_1 (\la), V_1(\la), U_2 (\la)$ and $V_2(\la)$ in \eqref{eq_linearization} and \eqref{eq_strong_lin}.\footnote{Nevertheless, we emphasize that a recent simple alternative to prove that a certain pencil is a (strong) linearization of a polynomial matrix is to prove that such a pencil is one of the (strong) block minimal bases pencils introduced in \cite{BKL}. See \cite{fiedler}.}
	
	However, already in the seventies Rosenbrock \cite{Rosen70} introduced the notion of \textit{polynomial system matrix} $S(\la)$ of a rational matrix $G(\la)\in\FF(\la)^{p\times m}$.
	That is, a matrix polynomial of the form
	\begin{equation}\label{eq_pol_sys_mat}
	S(\la):=\begin{bmatrix}
	A(\la) & B(\la)\\
	-C(\la) & D(\la)
	\end{bmatrix}  \in \FF [\la]^{(n+p) \times (n+m)}
	\end{equation}
	with $A(\la)\in\FF[\la]^ {n\times n}$ regular, such that its Schur complement with respect to $A(\la)$ is $G(\la)$, i.e., $$G(\la)= D(\la) + C(\la)A(\la)^{-1}B(\la).$$ The rational matrix $G(\la)$ is called the transfer function matrix of $S(\la)$ and the matrix polynomial $A(\la)$ is called the state matrix of $S(\la).$ Although the state matrix $A(\la)$ will appear in the $(1,1)$-block of $S(\la)$ in the theoretical results of this paper, we emphasize that in practice it can be at any place in $S(\la)$. The important property is that $G(\la)$ is the Schur complement of $A(\la)$ in $S(\la)$.
	
The importance of polynomial system matrices is related to the fact that under the minimality conditions described below, Rosenbrock showed that polynomial system matrices contain the pole and zero information of their transfer function matrices \cite{Rosen70}. Poles and zeros of rational matrices are defined through the notion of the Smith--McMillan form, that we state in what follows (see \cite{McMi2}, or \cite{Rosen70} for a more recent reference).

	\begin{deff}[Smith--McMillan form]
		For any rational matrix $G(\la)\in\efe(\la)^{p\times m}$ there exist unimodular matrices $U_{1}(\la)\in\efe[\la]^{p\times p}$ and $U_{2}(\la)\in\efe[\la]^{m\times m}$ such that
		\begin{equation}\label{eq:globSM}
		U_{1}(\la)G(\la)U_{2}(\la)=\left[\begin{array}{cc}
		\diag\left(\dfrac{\epsilon_1(\la)}{\psi_1(\la)},\ldots, \dfrac{\epsilon_r(\la)}{\psi_r(\la)}\right)&0 \\
		0& 0_{(p-r)\times (m-r)}
		\end{array}\right],
		\end{equation}
		where $r$ is the normal rank of $ G(\la)$ and, for $i=1,\ldots, r$, $\dfrac{\epsilon_i(\la)}{\psi_i(\la)}$ are nonzero irreducible rational functions with $\epsilon_i(\la)$ and $\psi_i(\la)$ monic polynomials (i.e., with leading coefficient equal to 1) that satisfy the divisibility chains $\epsilon_1(\la)\mid\cdots\mid\epsilon_r(\la)$ and $\psi_r(\la)\mid\cdots\mid\psi_1(\la)$. The diagonal matrix in (\ref{eq:globSM}) is called the Smith--McMillan form of $G(\la)$.
	\end{deff}
	
The rational functions $\epsilon_i(\la)/\psi_i(\la)$ in \eqref{eq:globSM} are called the invariant rational functions of $G(\la)$ and the finite poles and zeros of $G(\la)$ are the roots in $\overline\F$ of the denominators and numerators of the invariant rational functions, respectively.  In addition, for any $\la_0\in\overline\F$ each numerator $\epsilon_i(\la)$ can be factored as $\epsilon_i(\la)=(\la-\la_0)^{\alpha_i}p_{i}(\la)$ with $p_{i}(\la_0)\neq 0$ and $\alpha_i\geq 0$. The factors $(\la-\la_0)^{\alpha_i}$ with $\alpha_i\neq 0$ are called the zero elementary divisors at the finite zero $\la_0$. Analogously, for any $\la_0\in\overline\F$ each denominator $\psi_i(\la)$ can be factored as $\psi_i(\la)=(\la-\la_0)^{\beta_i}q_{i}(\la)$ with $q_{i}(\la_0)\neq 0$ and $\beta_i\geq 0$. The factors $(\la-\la_0)^{\beta_i}$ with $\beta_i\neq 0$ are called the pole elementary divisors  at the finite pole $\la_0$. If $G(\la)$ is a polynomial matrix then $\psi_i(\la)=1,$ for $i=1,\ldots, r$, and the diagonal matrix in (\ref{eq:globSM}) is just called the Smith form of $G(\la)$. In this polynomial case, there are no finite poles, the finite zeros are usually called finite eigenvalues, and the zero elementary divisors at a finite eigenvalue are just called elementary divisors at that eigenvalue.
	
We now introduce the notion of minimality of a polynomial system matrix, in order to relate its zeros with the poles and zeros of its transfer function matrix. A polynomial system matrix $S(\la)$ as in \eqref{eq_pol_sys_mat} is said to be minimal if
	\begin{equation}  \label{minimal}
	\rank\left[\begin{array}{ccc} A(\la_0) & B(\la_0)  \end{array}\right] = \rank \left[\begin{array}{ccc} A(\la_0) \\ -C(\la_0)   \end{array}\right]=n
	\end{equation}
	for all $\la_0\in\overline{\F}$. Then, Rosenbrock proved the following result \cite{Rosen70} about the recovery of the pole and zero information of a rational matrix from a minimal polynomial system matrix.
	
	\begin{theo}\label{th:Rosen} Let $S(\la)$ as in \eqref{eq_pol_sys_mat} be a minimal polynomial system matrix, with state matrix $A(\la)$, whose transfer function matrix is $G(\la).$  Then the elementary divisors at the finite eigenvalues of $A(\la)$ are the pole elementary divisors at the finite poles of $G(\la)$, and the elementary divisors at the finite eigenvalues of $S(\la)$ are the zero elementary divisors  at the finite zeros of $G(\la)$.
	\end{theo}

Theorem \ref{th:Rosen} is the foundation of the algorithms for solving numerically rational eigenvalue problems via ``linearization'', i.e., by constructing linear minimal polynomial system matrices of a given rational matrix $G(\la)$ and applying generalized eigenvalue algorithms to such system matrices. This approach is well known at least since the early eighties \cite{vandoorenIEEE} and has also been used in the specific case in which $G(\la)$ is a polynomial matrix \cite{vandooren-dewilde}. However, in these references, as well as in others using that approach, no attempt was made for relating linear minimal Rosenbrock's  polynomial system matrices of a polynomial matrix to the definition of linearization given by Gohberg, Lancaster and Rodman.
	
In this context a key remark is that in the particular case of $A(\la)$ in \eqref{eq_pol_sys_mat} being unimodular, $G(\la)$ has no finite poles, i.e., $G(\la)$ is a polynomial matrix, and condition \eqref{minimal} is satisfied. In addition, we have the following unimodular equivalence:
\begin{equation}\label{eq_unimodular}
 \underbrace{\begin{bmatrix}
 	C(\la)A(\la)^{-1} & I_p\\
 	A(\la)^{-1} & 0
 	\end{bmatrix}}_{\mbox{unimodular}} \begin{bmatrix}
A(\la) & B(\la)\\
-C(\la) & D(\la)
\end{bmatrix} \underbrace{\begin{bmatrix}
-A(\la)^{-1} B(\la) & I_n\\
I_m & 0
\end{bmatrix}}_{\mbox{unimodular}}=\begin{bmatrix}
G(\la) & \\
& I_n
\end{bmatrix}.
\end{equation}
Therefore, if $S(\la)$ is a pencil and $A(\la)$ is unimodular then $S(\la)$ is a linearization for the matrix polynomial $G(\la)$ according to Gohberg, Lancaster and Rodman's definition.

%=
%\underbrace{
%	\begin{bmatrix}
%	A(\la) & \\
%	-C(\la) & I_p
%	\end{bmatrix}}_{\mbox{unimodular}}
%\,
%\begin{bmatrix}
%I_n & \\
%& G(\la)
%\end{bmatrix}
%\,
%\underbrace{
%	\begin{bmatrix}
%	I_n &  A( \la)^{-1}  B(\la)\\
%	& I_m
%	\end{bmatrix}}_{\mbox{unimodular}}.

The purpose of this paper is to show that many of the linearizations for matrix polynomials in the literature are actually linear polynomial system matrices with unimodular state matrices of the corresponding matrix polynomial. Moreover, this property establishes a new tool to determine if a pencil $L(\la)$ is a linearization of a matrix polynomial. Namely, by computing the transfer function matrix of $L(\la)$ (i.e., the Schur complement) with respect to any unimodular submatrix $A(\la)$ that can be identified in the pencil. We aim to analyze a sufficiently large number of families of linearizations with this approach to convince the reader of its interest. However, we remark that it is not always possible to identify a unimodular submatrix $A(\la)$ inside a linearization such that the Schur complement with respect to $A(\la)$ is the corresponding polynomial matrix. Thus, it is not always possible to use this approach in the study of linearizations. We also emphasize that we are not stating that the new tool is better than previous techniques. But sometimes it may be simpler because the unimodular submatrix can be easily identified and the computation of the corresponding transfer function matrix is straightforward.

This paper is organized as follows.	In Section \ref{sec_frobenius}, we first revisit from the point of view of Rosenbrock's system matrices the Frobenius companion form or the (first) companion form \cite{GohbergKaashoekRodman88, GohbergLancasterRodman09}, that is one of the most classic linearizations.  Section \ref{sec_aux_results} includes some general auxiliary results about polynomial system matrices of matrix polynomials. Moreover, we discuss in Section \ref{sec_aux_results} how to recover easily the eigenvectors of a matrix polynomial from those of a linearization when it is viewed as a polynomial system matrix. Then, we consider in Section \ref{sec_comrade} the family of ``comrade'' linearizations \cite{aclan-2009,barnett,ortho}, that are particular cases of the CORK linearizations \cite{CORK} studied in Section \ref{sec_CORK}. We also analyze the family of (extended) block Kronecker linearizations \cite{BKL} in Sections \ref{sec_BK} and \ref{sec_EBK},  which includes, modulo permutations, all the families of Fiedler linearizations \cite{fiedler}.  We present a note in Section \ref{sec_rational} on how to use the Rosenbrock's system matrix approach to linearizations of polynomial matrices to construct easily linearizations for rational matrices, a task that has not been always easy via other approaches. Finally, some conclusions and possible lines of future research are presented in Section \ref{sec_conclusions}.

\section{Frobenius companion form}\label{sec_frobenius}
 Given a matrix polynomial $P(\la) \in \efe [\la]^{p\times m}$ written in terms of the monomial basis as in \eqref{eq:matrix poly}, the Frobenius companion form is the following pencil
\[
C_1(\lambda) :=
\left[
\begin{array}{ccccc}
\la P_k + P_{k-1} & P_{k-2} & \cdots & P_1 & P_0 \\
-I_m  & \la I_m \\ & \ddots & \ddots\\
& & \ddots & \la I_m\\
& & & -I_m & \la I_m
\end{array}
\right].
\]
One of the most basic results in the theory of matrix polynomials is that $C_1(\lambda)$ is a strong linearization of $P(\la)$ \cite{GohbergKaashoekRodman88}. The classical proofs of this result proceed by finding the four unimodular matrices that satisfy \eqref{eq_linearization} and \eqref{eq_strong_lin}. We now show that $C_1(\lambda)$ can be seen as a polynomial system matrix of $P(\la)$ with unimodular state matrix, which in turn implies that $C_1(\lambda)$ is a linearization of $P(\la)$.
\subsection{Frobenius companion form as a Rosenbrock's system matrix}
 We consider the following partition:
\begin{equation} \label{eq_frobenius}
C_1(\lambda) =
\left[
\begin{array}{cccc|c}
\la P_k + P_{k-1} & P_{k-2} & \cdots & P_1 & P_0 \\ \hline
-I_m  &  \la I_m &  &  & \\  &  \ddots &  \ddots &  & \\
&  & \ddots &  \la I_m &\\
& &  &  -I_m & \la I_m
\end{array}
\right] =: \begin{bmatrix}
-C(\la) & D(\la)\\
A(\la) & B(\la)
\end{bmatrix},
\end{equation}
as a Rosenbrock's system matrix with state matrix $A(\la)$. Then, $A(\la)$ is clearly unimodular and the transfer function matrix is
\begin{align}\label{eq_transfer_frobenius}
D(\la) + C(\la)A(\la)^{-1}B(\la)
= P_0 - \begin{bmatrix}
\la P_k + P_{k-1} & P_{k-2} & \cdots & P_1
\end{bmatrix}A(\la)^{-1}B(\la).
\end{align}	
To compute \eqref{eq_transfer_frobenius}, we consider the polynomial vector containing the elements of the monomial basis. Namely,
\begin{equation} \label{eq.defcapitalLambda}
\Lambda_{k-1}(\lambda):=\begin{bmatrix}
\la^{k-1} &
\la^{k-2}  &
\cdots &
\la  &
1
\end{bmatrix}^{T}.
\end{equation}
Then, observe that
\begin{equation} \label{eq.linearrelFRO}
\begin{bmatrix}
A(\la) & B(\la)
\end{bmatrix} (\Lambda_{k-1}(\lambda)\otimes I_m) = 0.
\end{equation}
Therefore, $A(\la)
\begin{bmatrix}
\la^{k-1} \, I_m &
\la^{k-2} \, I_m &
\cdots &
\la I_m
\end{bmatrix}^{T} + B(\la) = 0$ and
\begin{equation}\label{eq_invab}
A(\la)^{-1} B(\la) = -
\begin{bmatrix}
\la^{k-1} \, I_m &
\la^{k-2} \, I_m &
\cdots &
\la I_m
\end{bmatrix}^{T}.	
\end{equation}
Finally, by \eqref{eq_transfer_frobenius} and \eqref{eq_invab}, we obtain that the transfer function matrix of \eqref{eq_frobenius} is $P(\la)$.
\subsection{Reversal of the Frobenius companion form as a Rosenbrock's system matrix}
Now, we consider the reversal $\rev_1 C_{1}(\la)$ and the following partition:
\begin{equation} \label{eq_frobenius_rev}
\rev_1 C_1(\lambda) =
\left[
\begin{array}{c|cccc}
 P_k + \la P_{k-1} & \la P_{k-2} & \cdots & \la P_1 & \la P_0 \\ \hline
-\la I_m  &   I_m &  &  & \\  &  \ddots &  \ddots &  & \\
&  & \ddots &   I_m &\\
& &  &  -\la I_m &  I_m
\end{array}
\right] =: \begin{bmatrix}
D_r(\la) & -C_r(\la) \\
B_r(\la) & A_r(\la)
\end{bmatrix}.
\end{equation}
Then, we have that $A_r(\la)$ is unimodular, and the transfer function matrix of $\rev_1 C_1(\lambda)$ with the partition in \eqref{eq_frobenius_rev} is
\begin{align*}\label{eq_transfer_rev_frobenius}
D_r(\la) + C_r(\la)A_r(\la)^{-1}B_r(\la)
= P_k + \la P_{k-1} - \begin{bmatrix}
  \la P_{k-2} & \cdots & \la P_1 & \la P_0
\end{bmatrix}A_r(\la)^{-1}B_r(\la).
\end{align*}	
Taking into account that
\begin{equation} \label{eq.linearrelREVFRO}
\begin{bmatrix}
 B_r(\la) & A_r(\la)
\end{bmatrix} (\rev_{k-1}\Lambda_{k-1}(\lambda)\otimes I_m) = 0,
\end{equation}
we obtain that
\begin{equation*}\label{eq_invab_rev}
A_r(\la)^{-1} B_r(\la) = -
\begin{bmatrix}
\la I_m &
\cdots &
\la^{k-2} \, I_m &
\la^{k-1} \, I_m
\end{bmatrix}^{T}.	
\end{equation*}
Therefore, the transfer function matrix is
$$D_r(\la) + C_r(\la)A_r(\la)^{-1}B_r(\la)= P_k + \la P_{k-1} + \la^2 P_{k-2} + \cdots + \la^{k}P_0= \rev_k P(\lambda).$$ It then follows that $C_{1}(\la)$ is a strong linearization of $P(\la)$.

Notice that although the expression of the transfer function matrix involves $A(\la)^{-1}$, we computed very easily the transfer functions of $C_1(\la)$ and of $\rev_1 C_1(\lambda)$ without computing explicitly $A(\la)^{-1}$.  For that, we used the linear relation of the monomial basis in \eqref{eq.linearrelFRO} and \eqref{eq.linearrelREVFRO}. We will see that similar linear relations arise in other linearizations that can also be viewed as linear polynomial system matrices with unimodular state matrices $A(\la)$, and it allows to compute their transfer function matrices without computing explicitly $A(\la)^{-1}$.

%Notice that we computed the transfer function matrices of $C_{1}(\la)$ and %$\rev_1 C_1(\la)$ without computing the inverse of the state matrix $A(\la)$. In %general, although the computation of the transfer function matrix involves the %inverse of $A(\la)$, such computation is simple in most of the cases if there is %a linear relation satisfied by the considered polynomial basis as in the example %above. We will see more {\red such} examples in what follows.

\section{Auxiliary results: polynomial system matrices of matrix polynomials}
\label{sec_aux_results}
	
	The discussion in the introduction around \eqref{eq_unimodular} is summarized for the particular case of linear system matrices in the following Proposition \ref{prop_lin}. The proof follows from \eqref{eq_unimodular}.
	
		\begin{prop}\label{prop_lin}	Let $G(\la) \in\F(\la)^{p\times m}$ and let
			\begin{equation*}\label{eq_lin1}
			\mathcal{L}(\la)=\left[\begin{array}{cc}
			A_1 \la +A_0 &B_1 \la +B_0\\-(C_1 \la +C_0)&D_1 \la +D_0
			\end{array}\right]\in\F[\la]^{(n+p)\times (n+m)}
			\end{equation*} be a linear polynomial system matrix of $G(\la)$, with state matrix $A(\la):=A_1 \la +A_0$. If $A(\la)$ is unimodular then $G(\la)$ is a matrix polynomial and $	\mathcal{L}(\la)$ is a linearization of $G(\la)$.
		\end{prop}

In the next Proposition \ref{prop_unimod} we give a necessary and sufficient condition for the state matrix $A(\la)$ of a polynomial system matrix to be unimodular. This result can be useful in problems where the transfer function matrix $G(\la)$ is polynomial and computing $G(\la)$ is easier than  proving that $A(\la)$ is unimodular.
	
	\begin{prop}\label{prop_unimod}	Let $G(\la) \in\F(\la)^{p\times m}$ and let
		\begin{equation*}\label{eq_lin}
		S(\la)=\begin{bmatrix}
		A(\la) & B(\la)\\
		-C(\la) & D(\la)
		\end{bmatrix} \in\F[\la]^{(n+p)\times (n+m)}
		\end{equation*} be a polynomial system matrix of $G(\la)$, with state matrix $A(\la)$. Then $A(\la)$ is unimodular if and only if $	S(\la)$ is minimal and $G(\la)$ is a matrix polynomial.
	\end{prop}

\begin{proof} It is clear that if $A(\la)$ is unimodular then $G(\la)$ is a matrix polynomial and $S(\la)$ is minimal. That is, condition
	\eqref{minimal} holds. Conversely, if $G(\la)$ is a matrix polynomial and $S(\la)$ is minimal then $A(\la)$ has no finite eigenvalues, since the finite eigenvalues of $A(\la)$ would be the finite poles of $G(\la)$ by Theorem \ref{th:Rosen}. But $G(\la)$ has no finite poles since it is a matrix polynomial. Therefore, $\det A(\la)$ is constant, which implies that $A(\la)$ is unimodular.
\end{proof}

We could apply Proposition \ref{prop_lin} to the reversal $\rev_1\mathcal{L}(\la)$ to see that a linearization $\mathcal{L}(\la)$ is, in addition, a strong linearization.  For that, we may need to select another submatrix of $\rev_1\mathcal{L}(\la)$ as appropriate state matrix, with a different partition from the one considered in $\mathcal{L}(\la)$. However, selecting a unimodular submatrix in $\rev_1\mathcal{L}(\la)$ such that the Schur complement with respect to it is $\rev_{\ell} P(\la)$, for some $\ell\geq \deg P(\la)$, is not always possible. In that case, we can also try to apply the following Proposition \ref{prop_strong_lin}, which requires milder conditions.

\begin{prop}\label{prop_strong_lin} Let $\mathcal{L}(\la)$ be a linearization of a polynomial matrix $P(\la)\in\F[\la]^{p\times m}$. Assume that we can write
	\begin{equation*}\label{eq_lin_strong}
	\rev_1\mathcal{L}(\la)=\left[\begin{array}{cc}
	\widetilde A_1 \la + \widetilde A_0 &\widetilde B_1 \la +\widetilde B_0\\-(\widetilde C_1 \la +\widetilde C_0)&\widetilde D_1 \la +\widetilde D_0
	\end{array}\right]\in\F[\la]^{(n+p)\times (n+m)}
	\end{equation*} as a linear polynomial system matrix with state matrix $\widetilde A(\la):=\widetilde A_1 \la + \widetilde A_0$. If $ \widetilde A(\la)$ is invertible at $0$, i.e., if $\widetilde A_0$ is invertible, and the transfer function matrix of $\rev_1\mathcal{L}(\la)$ is equivalent at $0$ to $\rev_{\ell} P(\la)$, for some $\ell\geq \deg P(\la)$, then $	\mathcal{L}(\la)$ is a strong linearization of $P(\la)$.
\end{prop}
\begin{proof} Since $\mathcal{L}(\la)$ is a linearization of $P(\la)$, we know by \cite[Theorem 4.1]{spectral} that
	\begin{itemize}
		\item[\rm(a)] $\dim \mathcal{N}_r ( P)=\dim \mathcal{N}_r (\mathcal{L})$ and $\dim \mathcal{N}_{\ell} ( P)=\dim \mathcal{N}_{\ell} (\mathcal{L})$, where $\mathcal{N}_r (\cdot)$ and $\mathcal{N}_\ell (\cdot)$ stand for right and left rational null spaces, respectively.
		\item[\rm(b)] $\mathcal{L}(\la)$ and $P(\la)$ have exactly the same finite eigenvalues with the same elementary divisors.
	\end{itemize}
Moreover, taking into account that $\widetilde A_0$ is invertible and that the transfer function matrix of $\rev_1\mathcal{L}(\la)$ is equivalent at $0$ to $\rev_{\ell} P(\la)$, we have by \cite[Theorem 3.5]{local} that the elementary divisors at $0$ of $\rev_1\mathcal{L}(\la)$ and $\rev_{\ell} P(\la)$ are the same or, equivalently, that
	 	\begin{itemize}
	 	\item[\rm(c)] $\mathcal{L}(\la)$ and $P(\la)$ have exactly the same elementary divisors at infinity.
	 	\end{itemize}
Statements $\rm(a)$, $\rm(b)$ and $\rm(c)$ imply by \cite[Theorem 4.1]{spectral} that $\mathcal{L}(\la)$ is a strong linearization of $P(\la)$.
	\end{proof}

\subsection{Recovery of eigenvectors}

An advantage of considering linearizations of matrix polynomials as Rosenbrock's system matrices with unimodular state matrix is that the eigenvectors associated with an eigenvalue $\la_0$ can be recovered always in the same way. Given $\lambda_{0}\in\overline \F$ and a polynomial matrix $P(\la)\in\F[\la]^{p\times m}$, we consider the following vector spaces over $\overline \F$:
\[
\begin{array}{l}
\mathcal{N}_r (P(\la_0))=\{x\in\overline\FF^{m\times 1}: P(\la_0)x=0\}, \text{ and}\\
\mathcal{N}_\ell (P(\la_0))=\{y^T\in\overline\FF^{1\times p}: y^T P(\la_0)=0\},
\end{array}
\]
which are called, respectively, the right and left nullspaces over $\overline \F$ of $P(\la_0)$. If $\la_0$ is an eigenvalue of a regular $P(\la)$, then $\mathcal{N}_r (P(\la_{0}))$ and $\mathcal{N}_{\ell} (P(\la_{0}))$ are non trivial and contain, respectively, the right and left eigenvectors of $P(\la)$ associated with $\la_0.$

In the following Proposition \ref{prop:eigenvectors} we state the relation between the right and left nullspaces of a polynomial system matrix with unimodular state matrix and those of its polynomial transfer function matrix $P(\la)$. This is a particular case of the results from  \cite[Proposition 5.1]{DoMaQu19} and \cite[Proposition 5.2]{DoMaQu19}. Then, in the particular case of $P(\la)$ being regular, Proposition \ref{prop:eigenvectors} can be used to recover the right and left eigenvectors of $P(\la)$ from those of a polynomial system matrix of it with unimodular state matrix.

\begin{prop}\label{prop:eigenvectors} Let $S(\la)$ be a polynomial system matrix as in Proposition \ref{prop_unimod} with $A(\la)$ unimodular, and let $P(\la)\in\FF[\la]^{p\times m}$ be its transfer function matrix. Let $\la_0\in\overline \F$. Then, the following statements hold:
	\begin{itemize}
		\item[\rm(a)] The linear map
		\begin{align*}
		E_r \, : \, \mathcal{N}_r( P(\la_0)) & \longrightarrow \mathcal{N}_r(S(\la_0)) \\
		x & \longmapsto  \begin{bmatrix} -A(\lambda_0)^{-1}
		B(\lambda_0) \\I_m \end{bmatrix} x
		\end{align*}
		is a bijection between the right nullspaces over $\overline \F$ of $ P(\lambda_0)$ and $S(\lambda_0)$.
		
		\item[\rm(b)] The linear map
		\begin{align*}
		E_\ell \, : \, \mathcal{N}_\ell( P(\la_0)) & \longrightarrow \mathcal{N}_\ell(S(\la_0)) \\
		y^T & \longmapsto  y^T\begin{bmatrix}C(\lambda_0)A(\lambda_0)^{-1} & I_p \end{bmatrix}
		\end{align*}
		is a bijection between the left nullspaces over $\overline \F$ of $ P(\lambda_0)$ and $S(\lambda_0)$.
	\end{itemize}
\end{prop}

We can see that, in particular, right and left eigenvectors of a polynomial matrix $P(\la)$ can be directly recovered from the last block of the right and left eigenvectors of its polynomial system matrix $S(\la)$. In some cases, we can also recover the right or the left eigenvectors of $P(\la)$ from other blocks of the right or the left eigenvectors of $S(\la)$. This is illustrated for right eigenvectors in the following example.

\begin{example} Recall the Frobenius companion form $C_1(\la)$ in Section \ref{sec_frobenius} and the partition as a polynomial system matrix in \eqref{eq_frobenius}. By \eqref{eq_invab}, we have that, for any $\la_0\in\overline\F$, $$-A(\la_0)^{-1} B(\la_0) =
	\begin{bmatrix}
	\la_0^{k-1} \, I_m &
	\la_0^{k-2} \, I_m &
	\cdots &
	\la_0 I_m
	\end{bmatrix}^{T}.$$ Therefore, by Proposition \ref{prop:eigenvectors}, the linear map
	\begin{align*}
	F_r \, : \, \mathcal{N}_r( P(\la_0)) & \longrightarrow \mathcal{N}_r(C_1(\la_0)) \\
	x & \longmapsto  \begin{bmatrix}
	\la_0^{k-1} \, I_m &
	\la_0^{k-2} \, I_m &
	\cdots &
	\la_0 I_m & I_m
	\end{bmatrix}^{T} x
	\end{align*}
	is a bijection between the right nullspaces over $\overline \F$ of $ P(\lambda_0)$ and $C_1(\la_0)$.
\end{example}

\section{Comrade linearizations} \label{sec_comrade}

Consider a polynomial matrix
	$$
	P(\lambda)=P_{k}\phi_{k}(\lambda)+P_{k-1}\phi_{k-1}(\lambda)+\cdots + P_{1}\phi_{1}(\lambda)+P_{0}\phi_{0}(\lambda) \in \FF[\la ]^{p\times m},
	$$
	written in terms of a polynomial basis satisfying a three-term recurrence relation of the form: 	
	\[
	\alpha_{j}\phi_{j+1}(\lambda)=(\lambda-\beta_{j})\phi_{j}(\lambda)-\gamma_{j}\phi_{j-1}(\lambda) \quad j\geq 0
	\]
	where $\alpha_{j},\beta_{j},\gamma_{j}\in\FF,$ $ \alpha_{j}\neq 0,$ $\phi_{-1}(\lambda)=0,$ and $\phi_{0}(\lambda)=1.$ It is ``well-known'' that the  following ``comrade'' companion matrix introduced in \cite[Chapter 5]{barnett}
	\[ {\small C_\phi (\la) =
		\left[ \begin{array}{cccccc}
		\dfrac{(\lambda-\beta_{k-1})}{\alpha_{k-1}}P_{k}+P_{k-1} & P_{k-2}-\dfrac{\gamma_{k-1}}{\alpha_{k-1}}P_{k} & P_{k-3} & \cdots & P_{1} & P_{0} \\
		-\alpha_{k-2} I & (\lambda -\beta_{k-2}) I & -\gamma_{k-2} I & & & \\
		& -\alpha_{k-3} I & (\lambda - \beta_{k-3}) I & -\gamma_{k-3} I & & \\
		& &\ddots &\ddots & \ddots &  \\
		& & &-\alpha_{1} I& (\lambda-\beta_{1}) I & -\gamma_{1} I \\
		& & & &-\alpha_{0} I& (\lambda-\beta_{0}) I
		\end{array}  \right]}
	\]
is a strong linearization of $P(\la)$. Different proofs of this fact can be found in \cite{aclan-2009,DoMaQu19,ortho}. As we discuss in this section, this can also be proved via Rosenbrock's system matrices.

\subsection{Comrade linearizations as Rosenbrock's system matrices}
With the following partition:
	
	{\small \begin{align*}
		C_\phi (\la) & =
		\left[ \begin{array}{ccccc|c}
		\dfrac{(\lambda-\beta_{k-1})}{\alpha_{k-1}}P_{k}+P_{k-1} & P_{k-2}-\dfrac{\gamma_{k-1}}{\alpha_{k-1}}P_{k} & P_{k-3} & \cdots & P_{1} & P_{0} \\ \hline
		-\alpha_{k-2} I &  (\lambda -\beta_{k-2}) I &   -\gamma_{k-2} I & &  & \\
		&  -\alpha_{k-3} I &  (\lambda - \beta_{k-3}) I &  -\gamma_{k-3} I &  & \\
		&  & \ddots & \ddots & \ddots &  \\
		& & &-\alpha_{1} I&  (\lambda-\beta_{1}) I & -\gamma_{1} I \\
		& &  &  & -\alpha_{0} I& (\lambda-\beta_{0}) I
		\end{array} \right] \\ &  =:\begin{bmatrix}
		-C(\la) & D(\la)\\
		A(\la) & B(\la)
		\end{bmatrix},
		\end{align*}}
	we get that $C_\phi (\la)$ is a linear polynomial system matrix with unimodular state matrix $A(\la)$ and transfer function matrix $P(\la)$. Then $C_\phi (\la)$ is a linearization of $P(\la)$ by Proposition \ref{prop_lin}. Notice that comrade linearizations are constructed by considering the recurrence relation satisfied by the polynomial basis. They are particular cases of the more general notion of CORK linearizations, described in Section \ref{sec_CORK}. The computation of the transfer function matrix of $C_\phi (\la)$ is a particular case of the computation in the proof of Theorem \ref{theo_Corklin}. Therefore, we omit the details here and just emphasize that such computation is based on the identity
\[
\begin{bmatrix}
  A(\la) & B(\la)
\end{bmatrix} \, \left(
\begin{bmatrix}
  \phi_{k-1} (\la) \\
  \phi_{k-2} (\la) \\
  \vdots \\
   \phi_{0} (\la) \\
\end{bmatrix} \otimes I_m \right)
= 0 \, .
\]

	\subsection{Reversal of comrade linearizations as Rosenbrock's system matrices}
	
	To see that $C_\phi (\la)$ is a strong linearization, it is not possible to identify a unimodular submatrix of $ \rev_1 C_\phi (\la)$ such that the transfer function matrix is $ \rev_k P(\la)$. However, we can use Proposition \ref{prop_strong_lin}. For that, we consider the following partition:
		{\small \begin{align*}
	\rev_1	C_\phi (\la) & =
		\left[ \begin{array}{c|ccccc}
		\dfrac{(1-\lambda\beta_{k-1})}{\alpha_{k-1}}P_{k}+\lambda P_{k-1} & \lambda P_{k-2}-\lambda\dfrac{\gamma_{k-1}}{\alpha_{k-1}}P_{k} & \lambda P_{k-3} & \cdots & \lambda P_{1} & \lambda P_{0} \\ \hline
		-\lambda \alpha_{k-2} I &  (1 -\lambda \beta_{k-2}) I &   -\lambda \gamma_{k-2} I & &  & \\
		&  -\lambda \alpha_{k-3} I &  (1 - \lambda \beta_{k-3}) I &  -\lambda \gamma_{k-3} I &  & \\
		&  & \ddots & \ddots & \ddots &  \\
		& & &-\lambda \alpha_{1} I&  (1-\lambda\beta_{1}) I & -\lambda\gamma_{1} I \\
		& &  &  & -\lambda\alpha_{0} I& (1-\lambda\beta_{0}) I
		\end{array} \right] \\ &  =:\begin{bmatrix}
		\widetilde D(\la) & -\widetilde C(\la)\\
		\widetilde B(\la) & \widetilde A(\la)
		\end{bmatrix},
		\end{align*}}
	so that $\rev_1	C_\phi (\la)$ is a linear polynomial system matrix of $\dfrac{1}{f(\la)}\rev_{k} P(\la)$, with $f(\la):=\la^{k-1}\phi_{k-1}(1/\la)$, and state matrix $\widetilde A(\la)$. In addition, $\widetilde A(\la)$ is invertible at $0$ and $f(0)\neq 0$ since $\deg \phi_{k-1}(\la)=k-1$. The computation of the transfer function matrix of $\rev_1C_\phi (\la)$ is a particular case of the computation given in the proof of Theorem \ref{theo_cork_inf}. Therefore, we omit the details here.

%	It is known that $ C_\phi (\la)$ is a strong linearization of $ P(\la)$ (see, for instance, \cite[Corollary 3.4]{DoMaQu19}).
%	However, not always a linearization of a matrix polynomial can be seen as a linear polynomial system matrix with unimodular state matrix. An example of this is the reversal $ \rev_1 C_\phi (\la)$. That is, in general it is not possible to identify a unimodular submatrix of $ \rev_1 C_\phi (\la)$ such that the transfer function matrix is $ \rev P(\la)$.

	\section{CORK linearizations}\label{sec_CORK}
	
	In this section we consider polynomial matrices $P(\la)$ written as
	\begin{equation}\label{polmatrix}
	P(\la)=\displaystyle\sum_{i=0}^{k-1}(A_{i}-\la B_{i})p_{i}(\la) \in \FF[\la]^{p \times m},
	\end{equation}
	where $p_{i}(\la)$ are scalar polynomials with $p_{0}(\la) \equiv 1$ and $A_i$, $B_i\in \FF^{p \times m}$. Define the polynomial vector $$p(\la):=[p_{k-1}(\la)\quad \cdots \quad p_{0}(\la)]^{T},$$ and assume that the polynomials $p_i(\la)$ satisfy a linear relation
	\begin{equation}\label{linearrelation}
	(X-\la Y)p(\la)=0,
	\end{equation}
	where $\rank(X-\la_0 Y)=k-1\text{  for all }\la_0\in\overline \F,$ and $X-\la Y$ has size $(k-1)\times k.$ Then the matrix pencil
	\begin{equation}\label{eq_corklin}C(\la)=\left[\begin{array}{c}\begin{array}{ccc}
	A_{k-1}-\la B_{k-1} & \cdots & A_{0}-\la B_{0}
	\end{array} \\ \hdashline[2pt/2pt] \phantom{\Big|}
	(X-\la Y)\otimes I_{m}
	\end{array}\right]  \end{equation}
	is called a CORK linearization of $P(\la)$ \cite{CORK}.  CORK linearizations have played a fundamental role in the development of the Compact Rational Krylov (CORK) algorithm \cite{CORK} for solving numerically large scale PEPs that arise as approximations of other nonlinear eigenvalue problems coming from real-world applications. Such PEPs may have large degrees and the advantage of the CORK algorithm over previous methods are that its computational and memory costs are essentially independent of the degree of the PEP.

We show in the following subsection that the CORK linearization $C(\la)$ can be seen as a linear polynomial system matrix of $P(\la)$ with unimodular state matrix.
	
	\subsection{CORK linearizations as Rosenbrock's system matrices}
	
	\begin{theo}\label{theo_Corklin} Let $P(\la)$ be a matrix polynomial as in \eqref{polmatrix} and consider the matrix pencil $C(\la)$ in \eqref{eq_corklin}. Consider the following partition
		$$ C(\la)=\left[\begin{array}{c|c}
		\begin{array}{ccc}
		A_{k-1}-\la B_{k-1} & \cdots & A_{1}-\la B_{1}
		\end{array} & A_{0}-\la B_{0} \\ \hline \phantom{\Big|}
		X_1(\la) & \begin{array}{c}
		X_2(\la) \end{array}
		\end{array}\right], $$	
		where $(X-\la Y)\otimes I_m =: \begin{bmatrix}
		X_{1}(\la) & X_2(\la)
		\end{bmatrix} $ and $X_1(\la)$ has size $(k-1)m\times (k-1)m$. Then, $C(\lambda)$ is a linear polynomial system matrix with state matrix $X_1(\la)$ and transfer function matrix $P(\la)$. In addition, $X_1(\la)$ is unimodular.
	\end{theo}

\begin{proof}By \eqref{linearrelation}, we have that $\begin{bmatrix}X_{1}(\la) & X_2(\la)\end{bmatrix}(p(\la)\otimes I_m)=0$ and, thus,
	\begin{equation}\label{eq_x1x2}
 X_1(\la) [p_{k-1}(\la) I_m\quad \cdots \quad p_{1}(\la) I_m]^{T}+ X_2(\la)=0,\end{equation} taking into account that $p_0(\la)=1$.  From \eqref{eq_x1x2} it follows that $X_{1}(\la)$ is regular. By contradiction, if $X_{1}(\la)$ is singular there exists a nonzero polynomial vector $w(\la)$ such that $ w(\la)^T X_{1}(\la)=0$ and, therefore, $ w(\la)^T X_{2}(\la)=0$ by \eqref{eq_x1x2}. Thus, $ w(\la)^T \begin{bmatrix} X_{1}(\la) & X_{2}(\la)\end{bmatrix}=0$. But this is a contradiction since $\begin{bmatrix} X_{1}(\la) & X_{2}(\la)\end{bmatrix}$ has full row normal rank. Then $C(\lambda)$ is a linear polynomial system matrix with state matrix $X_1(\la)$ and its transfer function matrix is
\begin{equation}\label{eq_transfer_c}
A_{0}-\la B_{0} -\begin{bmatrix}
A_{k-1}-\la B_{k-1} & \cdots & A_{1}-\la B_{1}
\end{bmatrix}X_1(\la)^{-1}X_2(\la).
\end{equation}
 By \eqref{eq_x1x2}, we have that
\begin{equation}\label{eq_1}
X_1(\la)^{-1}X_2(\la)=- [p_{k-1}(\la) I_m\quad \cdots \quad p_{1}(\la) I_m]^{T},
\end{equation} and, by \eqref{eq_transfer_c} and \eqref{eq_1}, we obtain that the transfer function matrix is $$A_{0}-\la B_{0} +\begin{bmatrix}
A_{k-1}-\la B_{k-1} & \cdots &  A_{1} - \la B_{1}
\end{bmatrix}\begin{bmatrix}p_{k-1}(\la) I_m \\ \vdots \\ p_{1}(\la) I_m\end{bmatrix}=P(\la).$$

In addition, the state matrix $X_1(\la)$ is unimodular. To see this, we consider the following pencil $$X(\la):=\begin{bmatrix}
X_1(\la) & X_{2}(\la) \\
I_{(k-1)m} & 0
\end{bmatrix}$$ as a polynomial system matrix with state matrix $X_1(\la)$. Then we have that $X(\la)$ is minimal, since $\rank(X-\la_0 Y)=k-1\text{  for all }\la_0\in\overline \F,$ and the transfer function matrix (i.e., $-X_1(\la)^{-1}X_2(\la)$) is a polynomial matrix by \eqref{eq_1}. Then, by Proposition \ref{prop_unimod}, $X_1(\la)$ is unimodular.
\end{proof}

Theorem \ref{theo_Corklin} together with Proposition \ref{prop_lin} implies that $C(\la)$ is a linearization of $P(\la)$.

	\subsection{Reversal of CORK linearizations as Rosenbrock's system matrices}

 By assuming extra conditions in \eqref{linearrelation}, it follows from Proposition \ref{prop_strong_lin} and the next Theorem \ref{theo_cork_inf} that $C(\la)$ is, in addition, a strong linearization by considering $\rev_1 C(\la)$ as a Rosenbrock's system matrix.

\begin{theo}\label{theo_cork_inf} Let $P(\la)$ be a matrix polynomial as in \eqref{polmatrix} and consider the matrix pencil $C(\la)$ in \eqref{eq_corklin}. Assume that the $(k-1) \times k$ matrix $Y$ in \eqref{eq_corklin} satisfies $\rank Y = k-1$ and that $\deg p_{k-1}(\la)=k-1$. Consider the following partition for $\rev_1 C(\la)$:
	$$ \rev_1 C(\la)=\left[\begin{array}{c|c}
	\la A_{k-1}- B_{k-1} & \begin{array}{ccc}
		\la A_{k-2}- B_{k-2}  & \cdots &  \la A_{0}- B_{0}
	\end{array}\\ \hline \phantom{\Big|}
	\rev_1 Y_1(\la) &
	\rev_1 Y_2(\la)
	\end{array}\right]. $$	
 Then $\rev_1 C(\la)$ is a linear polynomial system matrix with state matrix $\rev_1 Y_2(\la)$  of size $(k-1)m\times (k-1)m$ and transfer function matrix $\dfrac{1}{q(\la)}\rev_k P(\la)$, where $q(\la):=\rev_{k-1}p_{k-1}(\la)$ and $q(0)\neq 0$. In addition, $\rev_1 Y_2(\la)$ is invertible at $0$.
\end{theo}
\begin{proof} First, taking into account that $\begin{bmatrix}Y_{1}(\la) & Y_2(\la)\end{bmatrix}(p(\la)\otimes I_m)=0$, we have that $\begin{bmatrix}\rev_1 Y_{1}(\la) & \rev_1 Y_2(\la)\end{bmatrix}(\la^{k-1}p(1/\la)\otimes I_m)=0$ and, thus,	\begin{equation}\label{eq_y1y2}
q(\la)	\rev_1 Y_1(\la) +\rev_1 Y_2(\la) [\la^{k-1}p_{k-2}(1/\la) I_m\quad \cdots \quad \la^{k-1}p_{0}(1/\la) I_m]^{T}=0,\end{equation} where $q(\la):=\rev_{k-1}p_{k-1}(\la)=\la^{k-1}p_{k-1}(1/\la)$. From \eqref{eq_y1y2}, and the fact that the matrix $\begin{bmatrix} \rev_1 Y_{1}(0) & \rev_1 Y_{2}(0)\end{bmatrix}$ has full row rank since $Y$ has full row rank, it follows that $\rev_1 Y_2(\la)$ is invertible at $0$, i.e., that $\rev_1 Y_2(0)$ is invertible. By contradiction, if $\rev_1 Y_2(0)$ is not invertible, there exists a constant vector $w$ such that $w^{T}\rev_1 Y_2(0)=0$ and, by \eqref{eq_y1y2}, $w^{T}\rev_1 Y_1(0)=0$ since $q(0)\neq 0$. Therefore, $w^{T}\begin{bmatrix} \rev_1 Y_{1}(0) & \rev_1 Y_{2}(0)\end{bmatrix}=0$ and this is a contradiction since $\begin{bmatrix} \rev_1 Y_{1}(0) & \rev_1 Y_{2}(0)\end{bmatrix}$ has full row rank.

We now compute the transfer function matrix of $\rev_1 C(\la)$ as a linear polynomial system matrix with state matrix $\rev_1 Y_2(\la)$. That is,
\begin{equation}\label{eq_transfer_rev}
T(\la):=\la A_{k-1}- B_{k-1} -\begin{bmatrix}
\la A_{k-2}- B_{k-2} & \cdots & \la A_{0}- B_{0}
\end{bmatrix}(\rev_1 Y_2(\la))^{-1}	\rev_1 Y_1(\la) .
\end{equation}
By \eqref{eq_y1y2}, we know that
\begin{equation}\label{eq_y1y2_inv}
(\rev_1 Y_2(\la))^{-1}	\rev_1 Y_1(\la) =-\frac{1}{p_{k-1}(1/\la) } [p_{k-2}(1/\la) I_m\quad \cdots \quad p_{0}(1/\la) I_m]^{T}.\end{equation}
Combining \eqref{eq_transfer_rev} and \eqref{eq_y1y2_inv}, we obtain
\begin{equation}\label{eq_transfer_rev2}
T(\la):=\la A_{k-1}- B_{k-1} +\frac{1}{p_{k-1}(1/\la) }\begin{bmatrix}
\la A_{k-2}- B_{k-2} & \cdots & \la A_{0}- B_{0}
\end{bmatrix} \begin{bmatrix}p_{k-2}(1/\la) I_m\\ \vdots \\ p_{0}(1/\la) I_m\end{bmatrix}.
\end{equation}
Multiplying $T(\la)$ by $q(\la)$ we obtain
$$q(\la)T(\la)=\displaystyle\sum_{i=0}^{k-1}(\la A_{i}-B_{i})(\la^{k-1}p_{i}(1/\la))=\rev_k P(\la).$$
\end{proof}

\section{Block Kronecker linearizations}\label{sec_BK}

In this section, we consider the block Kronecker pencils introduced in \cite{BKL} and show that they can also be seen as Rosenbrock's system matrices with unimodular state matrix. Block Kronecker pencils are a wide family of pencils that contains, among many other pencils, the Frobenius companion form and, modulo permutations, the Fiedler pencils originally introduced in \cite{fiedler-original} for scalar polynomials and extended in \cite{greeks,dtdm-fiedler} to matrix polynomials. They also include, modulo permutations, the generalized Fiedler pencils introduced in \cite{greeks}. These inclusions were studied in \cite{fiedler}. Block Kronecker pencils are particular cases of a much wider class of pencils that are called strong block minimal bases pencils \cite{BKL}. Block Kronecker pencils have excellent properties. For instance, they can be used for solving numerically PEPs in a structurally backward stable manner \cite{BKL}.

\begin{deff}\label{def_bk}
	Let $\la M_1 + M_0$ be an arbitrary pencil. Any pencil of the form
	\[
	C_{K} (\la) =
	\left[
	\begin{array}{c|c}
	\lambda M_1+M_0  & L_{\eta}(\lambda)^T\otimes I_{p}\\\hline
	   L_{\epsilon}(\lambda)\otimes I_{m}&0
	\end{array}
	\right]
	\>,
	\]
	is called a block Kronecker pencil, where
	\[
	L_k(\lambda)  :=\begin{bmatrix}
	-1& \lambda  \\
	& -1 & \lambda  \\
	& & \ddots & \ddots \\
	& & & -1 & \lambda   \\
	\end{bmatrix}\in\FF[\lambda]^{k\times (k+1)}.
	\]
	The one-block row and one-block column cases are included, i.e., the second block row or the second block column can be empty.
\end{deff}

It was proved in \cite[Theorem 4.2]{BKL} that any block Kronecker pencil is a strong  linearization of the polynomial matrix
\begin{equation} \label{eq.plambdaforBKL}
P(\la):=(\Lambda_\eta(\la)^T\otimes I_p)(\la M_1 + M_0)(\Lambda_\epsilon(\la)\otimes I_m) \, \in \FF[\lambda]^{p\times m},
\end{equation}
where $\Lambda_k (\lambda)$ was defined in \eqref{eq.defcapitalLambda}. The proof in \cite{BKL} is a corollary of the theory of strong block minimal bases pencils. We will see in this section that another proof easily follows from the approach via Rosenbrock's system matrices.

\subsection{Block Kronecker linearizations as Rosenbrock's system matrices}

Observe that we can write

\begin{equation}\label{eq_epsilon}
 L_{\epsilon}(\lambda)\otimes I_{m}=  \left[\begin{array}{cccc|c}
-I_m& \lambda I_m & & &  \\
& -I_m & \lambda I_m & &  \\
& & \ddots & \ddots \\
& & & -I_m & \lambda  I_m
\end{array} \right] =:  \left[\begin{array}{c|c} A_{\epsilon,m}(\la) &  B_{\epsilon,m} (\la)
\end{array} \right] ,
\end{equation}
and $A_{\epsilon,m}(\la) $ is unimodular. Analogously,

\begin{equation}\label{eq_eta} L_{\eta}(\lambda)\otimes I_{p}= \left[\begin{array}{cccc|c}
-I_p& \lambda I_p & & & \\
& -I_p & \lambda I_p & &  \\
& & \ddots & \ddots \\
& & & -I_p & \lambda  I_p
\end{array} \right] =:  \left[\begin{array}{c|c}  A_{\eta,p}(\la) &   B_{\eta,p} (\la)
\end{array} \right] ,
\end{equation}
and $A_{\eta,p}(\la) $ is unimodular. Then, $C_{K}(\la)$ can be partitioned as:

\[
C_{K} (\la) =
\left[
\begin{array}{c;{2pt/2pt}c|c}
M_{11}(\la) & M_{12}(\la) &  A_{\eta,p}(\la)^{T}  \\ \hdashline[2pt/2pt]
M_{21}(\la) & M_{22}(\la) &  B_{\eta,p}(\la)^{T}  \\ \hline
 A_{\epsilon,m}(\la) &   B_{\epsilon,m}(\la) & 0
\end{array}
\right],
\]
and we set

\begin{align}\label{eq_partition1}
A(\la) &  := \begin{bmatrix}
M_{11}(\la)  &  A_{\eta,p}(\la)^{T} \\
A_{\epsilon,m}(\la)  & 0
\end{bmatrix}, \quad
B(\la):=\begin{bmatrix}
M_{12}(\la) \\  B_{\epsilon,m} (\lambda)
\end{bmatrix} , \\ \label{eq_partition11}
C(\la)& :=-\begin{bmatrix}
M_{21}(\la) &  B_{\eta,p}(\la)^{T}
\end{bmatrix}, \quad\text{and} \quad D(\la):= M_{22}(\la) .
\end{align}
Notice that $A(\la)$ is unimodular for any $M_{11}(\la)$. With the partition above, we  prove in Theorem \ref{theo_BK_Rosenb} that $C_{K} (\la)$ is a linear polynomial system matrix with unimodular state matrix $A(\la)$ and whose transfer function matrix is the matrix polynomial $P(\la)$ in \eqref{eq.plambdaforBKL}.

\begin{theo} \label{theo_BK_Rosenb}
	Let $C_{K}(\la)$ be a block Kronecker pencil as in Definition \ref{def_bk}. Then, the following statements hold:
	\begin{itemize}
		\item[\rm(a)] The submatrix $A(\la)$ of $C_{K}(\la)$ as in \eqref{eq_partition1} is unimodular.
		
		\item[\rm(b)] The Schur complement of $A(\la)$ in $C_{K}(\la)$ is the polynomial matrix $P(\la)$ in \eqref{eq.plambdaforBKL}.
	
		\item[\rm(c)] $C_{K} (\la)$ is a linearization of $P(\la)$.
	\end{itemize}
	
\end{theo}

\begin{proof} Statement $\rm (c)$ follows from $\rm(a)$ and $\rm(b)$  and Proposition \ref{prop_lin}. Since (a) is obvious, it only remains to prove $\rm(b)$. First, we write
	\begin{align*}
	A(\la)^{-1} & := \begin{bmatrix}
	0 & A_{\epsilon,m}(\la)^{-1}  \\
	A_{\eta,p}(\la)^{-T}   & -A_{\eta,p}(\la)^{-T} M_{11}(\la)A_{\epsilon,m}(\la)^{-1}
	\end{bmatrix} \\ &
	= \begin{bmatrix}
	I_{\epsilon m} & 0 \\
	0 & A_{\eta,p}(\la)^{-T}
	\end{bmatrix} \begin{bmatrix}
	0 & I_{\epsilon m} \\
	I_{\eta p} & - M_{11}(\la)
	\end{bmatrix} \begin{bmatrix}
	I_{\eta p} & 0 \\
	0 & A_{\epsilon,m}(\la)^{-1}
	\end{bmatrix}.
	\end{align*}
	Now, observe that $\begin{bmatrix}A_{\epsilon,m}(\la) & B_{\epsilon,m}(\la)\end{bmatrix}(\Lambda_{\epsilon}(\la)\otimes I_m)=0$ and, thus,
	$$ A_{\epsilon,m}(\la) (\lambda\Lambda_{\epsilon-1}(\la)\otimes I_m)+ B_{\epsilon,m}(\la)=0.$$ Therefore,
	\begin{equation*}
	A_{\epsilon,m}(\la)^{-1}B_{\epsilon,m}(\la)=- (\lambda\Lambda_{\epsilon-1}(\la)\otimes I_m).
	\end{equation*}
	Analogously,
	\begin{equation*}
	B_{\eta,p}(\la)^{T}A_{\eta,p}(\la)^{-T}=- (\lambda\Lambda_{\eta-1}(\la)^{T}\otimes I_p).
	\end{equation*}
	Thus, the transfer function matrix of $C_K(\la)$ with $A(\la)$ as state matrix is, taking into account \eqref{eq_partition1} and \eqref{eq_partition11},
	{\small\begin{align*}
		&D(\la) + C(\la)A(\la)^{-1}B(\la)  \\ & =
		M_{22}(\la)  - \begin{bmatrix}
		M_{21}(\la) &  B_{\eta,p}(\la)^{T}
		\end{bmatrix} \begin{bmatrix}
		I_{\epsilon,m} & 0 \\
		0 & A_{\eta,p}(\la)^{-T}
		\end{bmatrix} \begin{bmatrix}
		0 & I_{\epsilon m} \\
		I_{\eta p} & - M_{11}(\la)
		\end{bmatrix} \begin{bmatrix}
		I_{\eta p} & 0 \\
		0 & A_{\epsilon,m}(\la)^{-1}
		\end{bmatrix} \begin{bmatrix}
		M_{12}(\la) \\  B_{\epsilon,m} (\la)
		\end{bmatrix}  \\ & =
		M_{22}(\la) - \begin{bmatrix}
		M_{21}(\la) & - (\lambda\Lambda_{\eta-1}(\la)^{T}\otimes I_p)
		\end{bmatrix} \begin{bmatrix}
		0 & I_{\epsilon m} \\
		I_{\eta p} & - M_{11}(\la)
		\end{bmatrix} \begin{bmatrix}
		M_{12}(\la) \\
		- (\lambda\Lambda_{\epsilon-1}(\la)\otimes I_m)
		\end{bmatrix}  \\ & =
		M_{22}(\la) +  (\lambda\Lambda_{\eta-1}(\la)^{T}\otimes I_p)M_{12}(\la) + M_{21}(\la) (\lambda\Lambda_{\epsilon-1}(\la)\otimes I_m) + (\lambda\Lambda_{\eta-1}(\la)^{T}\otimes I_p)M_{11}(\la) (\lambda\Lambda_{\epsilon-1}(\la)\otimes I_m)  \\ & =
		(\Lambda_{\eta}(\la)^{T}\otimes I_p) \begin{bmatrix}
		M_{11}(\la) & M_{12}(\la) \\
		M_{21}(\la) & M_{22}(\la)
		\end{bmatrix} (\Lambda_{\epsilon}(\la)\otimes I_m) = P(\la).
		\end{align*}} \end{proof}
	
\subsection{Reversal of block Kronecker linearizations as Rosenbrock's system matrices}	

We can consider the following partition for $\rev_1 C_{K}(\la):$

\begin{equation}\label{eq_rev_BK}
\rev_1  C_{K} (\la) =
\left[
\begin{array}{c|cc}
\widehat{M}_{11}(\la) & \widehat{M}_{12}(\la) &  \widehat{B}_{\eta,p}(\la)^{T}  \\ \hline
\widehat{M}_{21}(\la) & \widehat{M}_{22}(\la) &  \widehat{A}_{\eta,p}(\la)^{T}  \\
\widehat{B}_{\epsilon,m}(\la) &   \widehat{A}_{\epsilon,m}(\la) & 0
\end{array}
\right]:=\begin{bmatrix}
D_r(\la) & -C_r(\la) \\
B_r(\la) & A_r(\la)
\end{bmatrix},
\end{equation}
as a linear polynomial system matrix with state matrix $A_{r}(\la)$, where

\begin{equation}\label{eq_rev_epsilon}
 \rev_1 L_{\epsilon}(\lambda)\otimes I_{m}=  \left[\begin{array}{c|cccc}
-\la I_m&  I_m  \\
& -\la I_m &  I_m  \\
& & \ddots & \ddots \\
& & & -\la I_m &   I_m  \\
\end{array} \right] =:  \left[\begin{array}{cc} \widehat{B}_{\epsilon,m}(\la) &   \widehat{A}_{\epsilon,m} (\la)
\end{array} \right] ,
\end{equation}
and
\begin{equation}\label{eq_rev_eta}
 \rev_1 L_{\eta}(\lambda)\otimes I_{p}=  \left[\begin{array}{c|cccc}
-\la I_p&  I_p \\
& -\la I_p &  I_p  \\
& & \ddots & \ddots \\
& & & -\la I_p &  I_p  \\
\end{array} \right] =:  \left[\begin{array}{cc} \widehat{B}_{\eta,p}(\la) &   \widehat{A}_{\eta,p} (\la)
\end{array} \right] .
\end{equation}
Then, we have the following result,  whose proof is completely analogous to that of Theorem \ref{theo_BK_Rosenb} and, so, is omitted.

\begin{theo} \label{theo_BK_Rosenb_rev}
	Let $C_{K}(\la)$ be a block Kronecker pencil as in Definition \ref{def_bk} and $P(\la)$ be the polynomial matrix in \eqref{eq.plambdaforBKL}. Then, the following statements hold:
	\begin{itemize}
		\item[\rm(a)] The submatrix $A_r(\la)$ of $\rev_1 C_{K}(\la)$ as in \eqref{eq_rev_BK} is unimodular.
		
		\item[\rm(b)] The Schur complement of $A_r(\la)$ in $\rev_1 C_{K}(\la)$ is $\rev_{\eta+\epsilon+1} P(\la).$
		
		\item[\rm(c)] $C_{K} (\la)$ is a strong linearization of $P(\la)$.
	\end{itemize}
	
\end{theo}

\section{Extended block Kronecker linearizations}\label{sec_EBK}

In this section we consider a more general version of the notion of block Kronecker linearization  \cite{fiedler}.
	
	\begin{deff}\label{def_extended_bk}
		Let $\la M_1 + M_0$ be an arbitrary pencil and $Y \in \FF^{\varepsilon m \times\varepsilon m}$ and $ Z \in \FF^{\eta p \times \eta p}$ be arbitrary constant matrices. Then any pencil of the form
		\[
		C_{EK} (\la) =
		\left[
		\begin{array}{c|c}
		 \lambda M_1+M_0  & ( Z  (L_{\eta}(\lambda)\otimes I_{p}))^T\\\hline
		 Y  ( L_{\epsilon}(\lambda)\otimes I_{m})&0
		\end{array}
		\right]
		\>
		\]
		is called an extended block Kronecker pencil. The one-block row and one-block column cases are also included, i.e., the second block row or the second block column can be empty. Note that if $Z=I_{\eta p}$ and $Y=I_{\varepsilon m}$ then $C_{EK} (\la)$ is just a block Kronecker pencil.
	\end{deff}
	
%	\begin{theo} \label{theo_EBK}
%		Any extended block Kronecker pencil $C_{EK} (\lambda)$ with $Y$ and $Z$ invertible is a strong linearization of the matrix polynomial
%		\[
%		P(\la) := (\Lambda_\eta(\lambda)^T\otimes I_p)  (\lambda M_1+M_0) (\Lambda_{\epsilon}(\lambda)\otimes I_m) \in \FF[\lambda]^{p\times m} \, ,
%		\]
%		where \[
%		\Lambda_k(\lambda)^T   :=
%		\begin{bmatrix}
%		\lambda^{k} & \lambda^{k-1}  & \cdots & \lambda  & 1
%		\end{bmatrix} \in \FF[\lambda]^{1\times (k+1)}.
%		\]
%	\end{theo}

Extended block Kronecker pencils were introduced in \cite[Section 3.3]{fiedler} because though the family of block Kronecker pencils contains, modulo permutations, Fiedler pencils \cite{greeks,dtdm-fiedler,fiedler-original} and generalized Fiedler pencils \cite{greeks}, it is not large enough to include the other families of Fiedler-like pencils (Fiedler pencils with repetition \cite{greeks2} and generalized Fiedler pencils with repetition \cite{large-vector-spaces}). It was proved in \cite{fiedler} that all the Fiedler-like linearizations available in the literature are included modulo permutations in the family of extended block Kronecker pencils (which in turn is included in the much wider family of strong block minimal bases pencils \cite{BKL}). The proofs of these results require to choose highly structured matrices $Y$ and $Z$ in $C_{EK} (\lambda)$ depending on the coefficients of the polynomial matrix to be linearized. The details can be found in \cite{fiedler}.

Note that $C_{EK} (\la)$ in Definition \ref{def_extended_bk} can be written as $$C_{EK} (\la) = \mathrm{diag}(I,Y) \,  C_{K} (\la) \, \mathrm{diag} (I,Z^T),$$ where $C_K (\la)$ is a block Kronecker pencil as in Definition \ref{def_bk}. Thus, if $Y$ and $Z$ are nonsigular, then we get immediately that $C_{EK} (\la)$ is a strong linearization of exactly the same matrix polynomial $P(\la)$ in \eqref{eq.plambdaforBKL} as $C_{K} (\la)$. Our purpose in this section is to prove that we can also write $C_{EK} (\lambda)$ as a polynomial system matrix with unimodular state matrix and transfer function matrix \eqref{eq.plambdaforBKL}.

\subsection{Extended block Kronecker linearizations as Rosenbrock's system matrices}

Recall \eqref{eq_epsilon} and \eqref{eq_eta}, and observe that

\[
 Y  ( L_{\epsilon}(\lambda)\otimes I_{m})=  \left[\begin{array}{c|c} Y A_{\epsilon,m}(\la) &  Y B_{\epsilon,m} (\la)
\end{array} \right] ,
\]
and $Y A_{\epsilon,m}(\la) $ is unimodular if $Y$ is invertible. Analogously,

\[
Z ( L_{\eta}(\lambda)\otimes I_{p})= \left[\begin{array}{c|c} Z A_{\eta,p}(\la) &  Z B_{\eta,p} (\la)
\end{array} \right] ,
\]
and $Z A_{\eta,p}(\la) $ is unimodular if $Z$ is invertible. Then, $C_{EK}(\la)$ can be partitioned as:

\[
C_{EK} (\la) =
\left[
\begin{array}{c;{2pt/2pt}c|c}
	M_{11}(\la) & M_{12}(\la) &  A_{\eta,p}(\la)^{T} Z^{T} \\ \hdashline[2pt/2pt]
	M_{21}(\la) & M_{22}(\la) &  B_{\eta,p}(\la)^{T} Z^{T} \\ \hline
	Y A_{\epsilon,m}(\la) &  Y B_{\epsilon,m}(\la) & 0
\end{array}
\right],
\]
and we set

	\begin{align}\label{eq_partition}
\widetilde	A(\la) &  := \begin{bmatrix}
	M_{11}(\la)  &  A_{\eta,p}(\la)^{T} Z^{T} \\
Y A_{\epsilon,m}(\la)  & 0
	\end{bmatrix}, \quad
	\widetilde B(\la):=\begin{bmatrix}
M_{12}(\la) \\  Y B_{\epsilon,m} (\la)
	\end{bmatrix} , \\
	\widetilde C(\la)& :=-\begin{bmatrix}
	M_{21}(\la) &  B_{\eta,p}(\la)^{T}Z^{T}
	\end{bmatrix}, \quad\text{and} \quad \widetilde D(\la):= M_{22}(\la) .
	\end{align}
Notice that $\widetilde A(\la)$ is unimodular if $Y$ and $Z$ are invertible, for any $	M_{11}(\la)$. With the partition above, we  prove in Theorem \ref{theo_EBK_Rosenb} that $C_{EK} (\la)$ is a linear polynomial system matrix with unimodular state matrix $\widetilde A(\la)$ and  whose transfer function matrix is $P(\la)$ in \eqref{eq.plambdaforBKL}.

\begin{theo} \label{theo_EBK_Rosenb}
	Let $C_{EK}(\la)$ be an extended block Kronecker pencil as in Definition \ref{def_extended_bk}. Assume that $Y$ and $Z$ are invertible. Then, the following statements hold:
	\begin{itemize}
		\item[\rm(a)] The submatrix $\widetilde A(\la)$ of $C_{EK}(\la)$ as in \eqref{eq_partition} is unimodular.
		
		\item[\rm(b)] The Schur complement of $\widetilde A(\la)$ in $C_{EK}(\la)$ is the polynomial matrix $P(\la)$ in \eqref{eq.plambdaforBKL}.
		\item[\rm(c)] $C_{EK} (\la)$ is a linearization of $P(\la)$.
	\end{itemize}

\end{theo}

\begin{proof} Statement $\rm (c)$ follows from $\rm(a)$ and $\rm(b)$ and Proposition \ref{prop_lin}. Since (a) is obvious, it only remains to prove $\rm(b)$. For that, we write the matrices in \eqref{eq_partition} as follows:	
	\begin{align*}\label{eq_partition2}
	\widetilde A(\la) & = \begin{bmatrix}
	I_{\eta p}  & 0\\
	0 &  Y
	\end{bmatrix}A(\la)\begin{bmatrix}
	I_{\epsilon m}  & 0\\
	0 &  Z^{T}
	\end{bmatrix}, \quad
\widetilde 	B(\la)=\begin{bmatrix}
	I_{\eta p}  & 0\\
	0 &  Y
	\end{bmatrix}B(\la) , \\
	\widetilde C(\la)& =C(\la)\begin{bmatrix}
	I_{\epsilon m}  & 0\\
	0 &  Z^{T}
	\end{bmatrix} , \quad\text{and} \quad \widetilde D (\la)= D(\la) ,
	\end{align*}
	where $A(\la)$, $B(\la)$, $C(\la)$ and $D(\la)$ are as in \eqref{eq_partition1} and \eqref{eq_partition11}. Then, the transfer function matrix of $C_{EK}(\la)$  with $\widetilde A (\la)$ as state matrix is:
\begin{align*}
&\widetilde D(\la) + \widetilde C(\la)\widetilde A(\la)^{-1}\widetilde B(\la)  \\ &= D(\la) + C(\la)\begin{bmatrix}
I_{\epsilon m}  & 0\\
0 &  Z^{T}
\end{bmatrix} \begin{bmatrix}
I_{\epsilon m}  & 0\\
0 &  Z^{-T}
\end{bmatrix}A(\la)^{-1} \begin{bmatrix}
I_{\eta p}  & 0\\
0 &  Y^{-1}
\end{bmatrix}\begin{bmatrix}
I_{\eta p}  & 0\\
0 &  Y
\end{bmatrix}B(\la) \\ &=
D(\la) + C(\la)A(\la)^{-1}B(\la)=P(\la),
\end{align*}
as in the proof of Theorem \ref{theo_BK_Rosenb}.
 \end{proof}

\subsection{Reversal of extended block Kronecker linearizations as Rosenbrock's system matrices}

We can consider the following partition for $\rev_1 C_{EK}(\la):$

\begin{equation}\label{eq_rev_EBK}
\rev_1  C_{EK} (\la) =
\left[
\begin{array}{c|cc}
\widehat{M}_{11}(\la) & \widehat{M}_{12}(\la) &  \widehat{B}_{\eta,p}(\la)^{T} Z^{T} \\ \hline
\widehat{M}_{21}(\la) & \widehat{M}_{22}(\la) &  \widehat{A}_{\eta,p}(\la)^{T} Z^{T} \\
Y \widehat{B}_{\epsilon,m}(\la) &  Y \widehat{A}_{\epsilon,m}(\la) & 0
\end{array}
\right]:=\begin{bmatrix}
\widetilde D_r(\la) & - \widetilde C_r(\la) \\
\widetilde B_r(\la) & \widetilde A_r(\la)
\end{bmatrix},
\end{equation}
as a linear polynomial system matrix with state matrix $\widetilde A_{r}(\la)$, where

\[
Y  ( \rev_1 L_{\epsilon}(\lambda)\otimes I_{m})= \left[\begin{array}{cc} Y \widehat{B}_{\epsilon,m}(\la) &  Y \widehat{A}_{\epsilon,m} (\la)
\end{array} \right] ,
\]
and

\[
Z ( \rev_1 L_{\eta}(\lambda)\otimes I_{p})=  \left[\begin{array}{cc} Z \widehat{B}_{\eta,p}(\la) &  Z \widehat{A}_{\eta,p} (\la)
\end{array} \right] ,
\]
by using the notation in \eqref{eq_rev_epsilon} and \eqref{eq_rev_eta}, respectively. Then, we have the following result, whose proof is omitted since it  easily follows from Theorem \ref{theo_BK_Rosenb_rev}.

\begin{theo} \label{theo_EBK_Rosenb_rev}
	Let $C_{EK}(\la)$ be an extended block Kronecker pencil as in Definition \ref{def_extended_bk} and $P(\la)$ be the polynomial matrix in \eqref{eq.plambdaforBKL}. Assume that $Y$ and $Z$ are invertible. Then, the following statements hold:
	\begin{itemize}
		\item[\rm(a)] The submatrix $\widetilde A_r(\la)$ of $\rev_1 C_{EK}(\la)$ as in \eqref{eq_rev_EBK} is unimodular.
		
		\item[\rm(b)]  The Schur complement of $\widetilde A_r(\la)$ in $\rev_1 C_{EK}(\la)$ is $\rev_{\eta+\epsilon+1} P(\la)$.
%where 	\[
%P(\la) := (\Lambda_\eta(\lambda)^T\otimes I_p)  (\lambda M_1+M_0)  %(\Lambda_{\epsilon}(\lambda)\otimes I_m) \in \FF[\lambda]^{p\times m}.
%\]		
		\item[\rm(c)] $C_{EK} (\la)$ is a strong linearization of $P(\la)$.
	\end{itemize}
	
\end{theo}

%$A_{r}(\la)$ is unimodular if $Y$ and $Z$ are invertible and the transfer function matrix of $\rev_1 C_{EK}(\la)$ with the partition above is $D_r(\la) + C_r(\la)A_r(\la)^{-1}B_r(\la)=\rev_{\eta+\epsilon+1} P(\la).$

%	\begin{rem}\rm
%Modulo permutations, extended block Kronecker linearizations include: Fiedler %pencils (FP), Fiedler pencils with repetitions (FPR), generalized Fiedler pencils %(GFP), generalized Fiedler pencils with repetitions (GFPR) \cite{fiedler} and all %the pencils in the canonical basis of $\mathbb{DL}(P)$ since they are FPR %\cite{fiedler_structured}.
%	\end{rem}

	\section{A note on the construction of linearizations for rational matrices from linear system matrices of their polynomial parts} \label{sec_rational}
	
	By the division algorithm for polynomials, any scalar rational function $r(\la)$ can be uniquely written as $r(\la) =p(\la) +r_{sp}(\la),$ where $p(\la)$ is a polynomial and $r_{sp}(\la)$ is a strictly proper rational function. That is, the degree of the denominator of $r_{sp}(\la)$ is strictly larger than the degree of its numerator. Therefore, any rational matrix $R(\la)$ can be expressed uniquely as
			\begin{equation} \label{eq.prspdecomposition}
			R(\la) = P(\la) + R_{sp} (\la),
			\end{equation}
			where $P(\la)$ is a polynomial matrix and $ R_{sp} (\la)$ is a strictly proper rational matrix. That is, the entries of $ R_{sp} (\la)$ are strictly proper rational functions.

Recently, several papers have investigated the problem of constructing a linearization of a rational matrix $R(\la)$ from a linearization of $P(\la)$ and a minimal state-space realization of $R_{sp} (\la)$ \cite{Rosen70}. See, for instance, \cite{alam1,alam2,strong,DoMaQu19,su-bai} among other references on this topic. Often, the development of such constructions has required considerable theoretical work, since it was not obvious how to merge the linearization of $P(\la)$ with the minimal state-space realization of $R_{sp} (\la)$. In contrast, if one considers linearizations of $P(\la)$ that are Rosenbrock's system matrices with unimodular state matrices and transfer function matrices equal to $P(\la)$, the construction is very easy. For this purpose, we consider on the one hand a linear Rosenbrock's system matrix
			\[
				L(\la)=\begin{bmatrix}
				A(\la) & B(\la)\\
				-C(\la) & D(\la)
				\end{bmatrix} ,
			\]
			with unimodular state matrix $A(\la)$, and transfer function matrix equal to $P(\la)$. That is,
		$$P(\la) = D(\la)+C(\la)A(\la)^{-1}B(\la).$$
		On the other hand, we consider a minimal state-space realization $$R_{sp} (\la) = C_s (\la I_s - A_s)^{-1} B_s$$ of the strictly proper part $ R_{sp} (\la)$. Then, we construct the following pencil:
			\[
			\mathcal{L}(\la)=\left[
			\begin{array}{cc|c}
			(\la I_s - A_s) & 0 & B_s \\
			0 & A(\la) & B(\la)\\ \hline
			-C_s &-C(\la) & D(\la)
			\end{array}
			\right].
			\]
We obtain that $\mathcal{L}(\la)$ is a linear minimal polynomial system matrix of $R(\la)$ in \eqref{eq.prspdecomposition}, with state matrix $\left[\begin{smallmatrix}
                                                  (\la I_s - A_s) & 0 \\
                                                  0 & A(\la)
                                                \end{smallmatrix} \right]$.
Thus, by Theorem \ref{th:Rosen}, $\mathcal{L}(\la)$ contains the information about finite poles and zeros of $R(\la)$, and is a linearization of $R(\la)$  in the sense of \cite[Definition 3.2]{strong}. More information about different definitions of linearizations of rational matrices and how to construct linear polynomial system matrices that also preserve the pole and zero information at infinity, i.e., the pole and zero information at $0$ of $R(1/\la)$, can be found, for instance, in \cite{strong,local,strongminimal,stronglyminimalSIMAX}.

\section{Conclusions and lines of future research} \label{sec_conclusions} The main message of this paper in that in the analysis of a pencil $L(\la)$ that may be a linearization of a matrix polynomial, it is worth looking for a unimodular submatrix $A(\la)$ of $L(\la)$ and viewing $L(\la)$ as a Rosenbrock's system matrix with state matrix $A(\la)$. In the first place, this allows to identify a matrix polynomial $P(\la)$ for which $L(\la)$ is a linearization just by computing the Schur complement of $A(\la)$ in $L(\la)$. Other advantages of this approach are that it may yield simplified alternative proofs that certain well known linearizations are indeed linearizations, it gives rise to simple recovery rules for eigenvectors and to simple constructions of linearizations of rational matrices whose polynomial part is $P(\la)$. This approach applied to the reversal of $L(\la)$ can also be used to prove that $L(\la)$ is a strong linearization of $P(\la)$, though it usually requires to identify a different submatrix in $\rev_1 L(\la)$. One possible line of future research in this context is to apply the Rosenbrock's system matrix approach in the study of $\ell$-ifications of polynomial matrices \cite{spectral,elification1,elification3,elification2}. Another connected problem that deserves to be investigated is that it may be possible to identify in a pencil several unimodular submatrices, possibly with different sizes (think, for instance, in the Frobenius companion form), which will give rise to several different matrix polynomials with the same eigenvalue information. Finally, we emphasize that this paper has connected two very important historical tools in the theory of rational and polynomial matrices: Rosenbrock's system matrices and Gohberg, Lancaster and Rodman's linearizations and strong linearizations.

%\section*{References}

\end{document}